\documentclass[11pt]{article}
\usepackage[table]{xcolor}
\usepackage{graphicx} 
\usepackage{comment}
\usepackage{amsmath,amsfonts,amssymb}
\usepackage{graphicx} \usepackage{multirow} \usepackage{tikz,pgf}
\usepackage{url} \usepackage{amsmath} \usepackage{graphicx}
\usepackage{epsfig} \usepackage{epstopdf} \usepackage{color}
\usepackage{enumitem} \usepackage{ amsfonts} \usepackage{amsmath}
\usepackage{tikz}
\usepackage{pgfmath}
\usepackage{tikz-3dplot}
\usetikzlibrary{math}
\usepackage{amsthm} \parindent 0em \parskip 0.5em
\usepackage{mdframed}
\newmdenv[linecolor=black,skipabove=\topsep,skipbelow=\topsep,
leftmargin=-3pt,rightmargin=-3pt,
innerleftmargin=6pt,innerrightmargin=6pt]{mybox}
 \def\tto{\;{\lower 1pt
		\hbox{$\rightarrow$}}\kern -10pt \hbox{\raise 2pt
		\hbox{$\rightarrow$}}\;}  
 \def\ra{\rangle} \def\la{\langle}
 \def\epsilon{\varepsilon} 
 \def\R{\Bbb R}

\def\t{\tau}

   \def\t{\tau}

\setlist[enumerate,1]{itemsep=0.0ex,parsep=0.5ex,label={\rm(\alph*)},leftmargin=*,
	align=left} \newcounter{lk}

\usepackage[colorlinks=true]{hyperref}

\begin{document} 
	\begin{center}
		{\sc\bf Qualitative Properties of $k-$Center 
			Problems}\\[1ex]
		{\sc V. S. T. Long}\footnote{Faculty of Mathematics and Computer
			Science, University of Science, Ho Chi Minh City,
			Vietnam.}$^,$\footnote{Vietnam National University, Ho Chi Minh
			City, Vietnam (email: vstlong@hcmus.edu.vn).}, 
			 {\sc N. M.  Nam}\footnote{Fariborz Maseeh Department of Mathematics and Statistics,
			Portland State University, Portland, OR 97207, USA (email: mnn3@pdx.edu). Research of this author was partly supported by the USA National Science Foundation under grant DMS-2136228.},
			{\sc J. Sharkansky$^3$, }
		{\sc N. D. Yen}\footnote{Institute of Mathematics, Vietnam Academy of Science and Technology, 18 Hoang Quoc Viet,
			Hanoi {\color{green} 10072}, Vietnam (email: ndyen@math.ac.vn).}
	\end{center}
	\small{\bf Abstract.} In this paper, we study generalized versions of the $k$-center problem, which involves finding \(k\) circles of the smallest possible equal radius that cover a finite set of points in the plane. By utilizing the Minkowski gauge function, we extend this problem to generalized balls induced by various convex sets in finite dimensions, rather than limiting it to circles in the plane. First, we establish several fundamental properties of the global  optimal solutions to this problem. We then introduce the notion of local  optimal solutions and provide a sufficient condition for their existence. We also provide several illustrative examples to clarify the proposed problems.
	\\[1ex]

	\newtheorem{theorem}{Theorem}[section]
	\newtheorem{proposition}[theorem]{Proposition}
 \newtheorem{lemma}[theorem]{Lemma}
	\newtheorem{corollary}[theorem]{Corollary}
	\theoremstyle{definition}
	\newtheorem{remark}[theorem]{Remark}
	\newtheorem{definition}[theorem]{Definition}
	\newtheorem{example}[theorem]{Example}
	
	\normalsize
	\section{Introduction and Preliminaries}

Given a finite set $A = \{a_1, \ldots, a_m\}$ of distinct demand points (also referred to as \textit{target points}) in $\mathbb{R}^d$, the task is to find $k$ balls with centers $x_1, \ldots, x_k$ and the smallest possible radius $r \geq 0$ such that the union of these balls covers all the demand points. This optimization problem is known as the \emph{$k$-center problem}. The $k$-center problem is a generalization of the \emph{smallest enclosing circle problem} introduced in the 19th century by the English mathematician James Joseph Sylvester (1814--1897) \cite{syl}, which asks for a circle of the smallest radius enclosing a given finite set of points in the plane. The smallest enclosing circle problem has been investigated from both theoretical and numerical aspects; see \cite{Beck,Brimberg,Gunther1,Gunther2,mordukhovich2011applications,nam2014constructions} and the references therein. Nevertheless, the qualitative properties of the $k$-center problem have not been extensively studied in the literature, due to its complexity as a nonsmooth and nonconvex optimization problem.

Let us abbreviate the index set of the demand/target points (resp., of the centers/facilities) by $I$ (resp., by $J$), that is, $I = \{1, \ldots, m\}$ and $J = \{1, \ldots, k\}$, and assume that $1 \leq k \leq m$. If the centers $x_1,\ldots,x_k$ have been chosen, each demand point $a_i$ is assigned to one of the nearest centers, denoted by $x_{\ell_i}$, which can be found by solving the discrete optimization problem $\min\big\{\|x_\ell-a_i\|\mid \ell\in J\big\}.$ Then, the ball centered at $x_{\ell_i}$ with the radius 
$$r_i= \min\limits_{\ell\in J}\|x_\ell-a_i\|$$ 
covers the demand/target point $a_i$. Therefore, the smallest radius $r \geq 0$ of the balls centered at $x_1,\ldots,x_k$, which satisfies the requirement \textit{``the union of those balls covers all the demand points''}, is given by the formula $r = \max\limits_{i\in I} r_i$. Thus, the $k$-center problem can be formulated as the \textit{unconstrained continuous optimization problem}
\begin{eqnarray}\label{maxmin}
	\text{min}\left\{f_k(x) = \max_{i\in I}\,\left(\min_{\ell\in J}\|x_\ell-a_i\|\right)\mid x=(x_1, x_2, \ldots, x_k)\in (\mathbb{R}^d)^k = \mathbb{R}^{dk}\right\}.
\end{eqnarray}
Note that the Euclidean norm in problem~\eqref{maxmin} can be replaced by the squared Euclidean norm to induce smoothness while producing an equivalent problem.

 Given a nonempty compact convex set $F\subset \mathbb R^d$ such that  $0\in \mbox{\rm int}(F)$, where $0$ signifies the origin of $\mathbb{R}^d$, we recall that the {\em Minkowski gauge function} $\rho_F\colon \R^d \to [0,\infty)$ associated with $F$ is defined by  (see~\cite[Section~6.3]{mordukhovich2023easy}):
		\begin{equation}\label{gauge}
			\rho_F(x)=\inf\{t\geq 0\; |\; x\in tF\},\ \; x\in \R^d.
		\end{equation}
		Then a more general model of~\eqref{maxmin} can be formulated as
		\begin{eqnarray}\label{maxminn}
			\text{min}\left\{f_k^F(x)=\max_{i\in I}\left(\min_{\ell\in J}\rho_F(x_\ell-a_i)\right) \mid x=(x_1, x_2, \ldots, x_k)\in \R^{dk}\right\}.
		\end{eqnarray}
Note that when $F$ is the closed unit ball of 
  $\R^d$ with respect to the Euclidean norm, then~\eqref{maxminn} reduces to~\eqref{maxmin}. The use of the Minkowski gauge function instead of the Euclidean norm allows us to study the $k$-center problems with balls generated by different norms and even generalized balls generated by convex sets.

With the presence of the ``min''  operator, the objective functions $f$ and $f_F$ 
 of problem~\eqref{maxmin} and problem~\eqref{maxminn}, respectively, are nonconvex and nonsmooth in general. This makes the problems challenging for available optimization techniques. The obvious representation
	\begin{equation*}\label{dc}
		\min_{\ell\in J}\rho_F(x_\ell-a_i)=\sum_{r\in J} \rho_F(x_r-a_i)- \max_{\ell\in J} \sum_{r\in J\setminus \{\ell\}}\rho_F(x_r-a_i)
	\end{equation*}
shows that the function $x=(x_1, x_2, \ldots, x_k)\mapsto \min_{\ell\in J}\rho_F(x_\ell-a_i)$ is the difference of two convex functions on $\R^{dk}$. Since the maximum of finitely many DC functions (\textbf{D}ifference-of-\textbf{C}onvex functions)  is a DC function (see, e.g.,~\cite[Proposition~4.1]{tuy2016convex}), we see that~\eqref{maxmin} and \eqref{maxminn} belong to the class of DC programming problems.

There are other well-known models of multifacility location, called \textit{the multi-source Weber problem} and \textit{the minimum sum-of-squares clustering problem}, whose aim is to seek $k$ centers/facilities $x_1,x_2, \ldots,x_k$ to serve~$m$ demand/target points $a_1,\ldots,a_m$ in $\R^d$ by assigning each of them to its nearest center and minimizing the sum of the weighted minima of the Euclidean distances or the squared Euclidean distances 
of the demand points to the facilities. Practical applications of these multifacility location problems include finding locations to place $k$ distribution centers to deliver supplies to $m$ demand points so that the total transportation cost is minimized. Recently, fundamental qualitative properties of the models have been established in \cite{cuongnca,cuongoptimization,cuong_jnca,cuong2020qualitative}.

Most of the existing numerical methods for solving problem~\eqref{maxmin}, which involves $k$-centers of Euclidean balls, rely on heuristic approaches (see~\cite{Brimberg,mordukhovich2019fermat,nam2014constructions} and the references therein). These methods often struggle to deliver optimal solutions in the vast majority of cases. The first significant effort to develop non-heuristic optimization algorithms, based on minimizing DC (Difference of Convex) functions, was initiated in \cite{ANQ}. However, comprehensive qualitative studies of both problem~\eqref{maxmin} and its generalization in~\eqref{maxminn} remain largely unexplored. The present paper aims to uncover some fundamental properties of the global and local  optimal solution sets of the $k$-center problem~\eqref{maxmin} and its more general model~\eqref{maxminn}.
	
The organization of the paper is as follows. Section~2 presents a global solution existence theorem and discusses the properties of the global  optimal solution set for problem~\eqref{maxminn}. In Section~3, we establish a sufficient condition for the existence of local  optimal solutions to problem~\eqref{maxminn}. Several concluding remarks are provided in Section~4.

Throughout this paper, the set of positive integers is denoted by~$\mathbb{N}$. The Euclidean space $\mathbb{R}^d$, where $d \in \mathbb{N}$, is equipped with the inner product $\langle x, y \rangle = \sum_{i=1}^d x_i y_i$ and the norm $\|x\| = \left(\sum_{i=1}^d x_i^2\right)^{1/2}$, where $x = (x_1, \ldots, x_d)$ and $y = (y_1, \ldots, y_d)$. The closed Euclidean ball  with center $a \in \mathbb{R}^d$ and radius $r \geq  0$ is denoted by $\mathbb B[a, r]$. Given a subset $K$ of $\R^d$, define $\|K\|=\sup\{\|x\|\; |\; x\in K\}$. We denote the sets of all optimal solutions to problem~\eqref{maxmin} and problem~\eqref{maxminn} by $S_k$ and $S_k^F$, respectively. 
	
	\section{Global Optimal Solutions}
	In this section, we study the existence of global optimal solutions and some properties to the global  optimal solution sets for the $k$-center problem~\eqref{maxmin} and the more general model~\eqref{maxminn}.

Although the following lemma is a known result, we provide a detailed proof for the reader's convenience.

	\begin{lemma}\label{lemma1} {\rm (See~\cite[Lemma 4.1]{longoptimletter})}
Let $F \subset \mathbb{R}^d$ be a nonempty compact convex set satisfying the condition $0 \in \mbox{\rm int}(F)$. Then, for any $x \in \mathbb{R}^d$ we have
\begin{equation}\label{estimate_1} 
\rho_F(x) \leq \|F\| \|F^{\circ}\| \rho_F(-x),
\end{equation}
where $F^{\circ} = \{x^* \in \mathbb{R}^d \mid \langle x^*, x \rangle \leq 1 \; \mbox{\rm for all } x \in F\}$.
\end{lemma}

	\begin{proof} For any $x\in \R^d$, by~\cite[Proposition~2.1(c)]{CW_JOGO2004} we have $$\dfrac{\rho_F (x)}{\|F^{\circ}\|}\leq \|x\|\leq\|F\|\rho_F(x).$$ Consequently,
	$$\rho_F (x)\leq \|F^{\circ}\|\|x\|=\|F^{\circ}\|\|-x\|\leq \|F\|\|F^{\circ}\|\rho_F(-x).$$ Thus, the estimate~\eqref{estimate_1} holds.	
	\end{proof}

	The following theorem generalizes~\cite[Proposition~2.2]{mordukhovich2019fermat}. To obtain it, we will use some arguments similar to those in the proofs of global  optimal solution existence in~\cite[Theorem~4.2]{ANQ} and \cite[Theorem~2.4]{cuong2020qualitative}.

	\begin{theorem}\label{thm_global}
		The global optimal solution set $S_k^F$ of problem~\eqref{maxminn} is nonempty and closed.
	\end{theorem}
	\begin{proof}
		Take any $i_0 \in I$ and put 
		\begin{equation}\label{eqrho}
         \rho=\max\limits_{i \in I}\rho_F(a_{i_0}-a_{i}).
		\end{equation}
		For every $\gamma\geq 0$, let
		\begin{equation}\label{ball_F}
\mathbb B_F[a_{i_0},\gamma]=\big\{y\in \mathbb R^d\mid \rho_F(y-a_{i_0})\leq \gamma\big\}.
		\end{equation} 
		Since $\rho_F$ is Lipschitz continuous (see~\cite[Proposition~6.18]{mordukhovich2023easy}), the function $$f_k^F(x)=\max_{i\in I}\, \min_{\ell\in J}\rho_F(x_\ell-a_i)$$ is also Lipschitz continuous. The set $\mathbb B_F[a_{i_0},\gamma]$, which contains $a_{i_0}$, is bounded for every $\gamma\geq 0$. Indeed, since $F$ is a bounded set, there exists $\alpha>0$ such that $F\subset \mathbb B[0,\alpha]$. Therefore, taking any $y\in\mathbb B_F[a_{i_0},\gamma]$, by~\eqref{gauge} and~\eqref{ball_F} we have
		$$\begin{array}{rcl}
		\gamma\geq \rho_F(y-a_{i_0}) & = &\inf\{t\geq 0\; |\; y-a_{i_0}\in tF\}\\
			& \geq & \inf\{t\geq 0\; |\; y-a_{i_0}\in t\mathbb B[0,\alpha]\}\\
			& = & \alpha^{-1} \|y-a_{i_0}\|.
		\end{array}$$ It follows that $\mathbb B_F[a_{i_0},\gamma]\subset \mathbb B[a_{i_0},\alpha\gamma]$. Thus, $\mathbb B_F[a_{i_0},\gamma]$ is bounded. In addition, from~\eqref{ball_F} 
	 we deduce that $\mathbb B_F[a_{i_0},\gamma]$ is closed as it is a sublevel set of a continuous function. Hence, $\mathbb B_F[a_{i_0},\gamma]$ is compact.  By the classical Weierstrass theorem (see, e.g., \cite[Theorem 7.9(i)]{mordukhovich2023easy}), we can infer that the optimization problem
		\begin{equation}\label{compactification}
		\min \left\{f_k^F(x) \mid x=\left(x_1, \ldots, x_k\right) \in \mathbb{R}^{d k},\; x_\ell \in \mathbb B_F\left[a_{i_0}, (1+\|F\|\|F^\circ\|) \rho\right],\; \forall \ell \in J\right\},
	    \end{equation} 
		admits a global  optimal solution $\bar{x}=\left(\bar{x}_1, \ldots, \bar{x}_k\right) \in \mathbb{R}^{d k}$ satisfying $$\rho_F(\bar{x}_\ell-a_{i_0}) \leq (1+\|F\|\|F^\circ\|) \rho\textrm{ for all }\ell \in J.$$ Given any $x=\left(x_1, \ldots, x_k\right) \in \mathbb{R}^{d k}$, we distinguish two situations: 
		
		\hskip0.5cm (a) $\rho_F(x_\ell-a_{i_0}) \leq (1+\|F\|\|F^\circ\|) \rho$ for all $\ell \in J$;
		
		\hskip0.5cm (b) $\rho_F(x_{ \ell}-a_{i_0}) > (1+\|F\|\|F^\circ\|) \rho$ for some $\ell \in J$. 
		
		In situation~(a), we have $f_k^F(\bar{x}) \leq f_k^F(x)$ because $x$ is a feasible point of the optimization problem~\eqref{compactification}. In situation~(b), denote the set of $\ell\in J$ with $\rho_F(x_{ \ell}-a_{i_0}) > (1+\|F\|\|F^\circ\|) \rho$ by~$L$. Let $\tilde{x}=\left(\tilde{x}_1, \ldots, \tilde{x}_k\right) \in \mathbb{R}^{d k}$ be the vector defined by setting $\tilde{x}_\ell=a_{i_0}$ for $\ell \in L$ and  $\tilde{x}_\ell=x_\ell$ for $\ell \in J \backslash L$. Fixing any $i \in I$,  one has  for every $\ell \in L$ the following:
		$$\begin{array}{ll}
        \rho_F(\tilde{x}_\ell-a_i)=\rho_F(a_{i_0}-a_i) \leq \rho &=(1+\|F\|\|F^\circ\|)\rho- \|F\|\|F^\circ\|\rho\\
        & <\rho_F(x_\ell-a_{i_0})- \|F\|\|F^\circ\|\rho\\
        & \leq \rho_F(x_\ell-a_{i_0})- \|F\|\|F^\circ\|\rho_F(a_{i_0}-a_i) \quad \text{(by \eqref{eqrho})} \\
        & \leq \rho_F(x_\ell-a_{i_0})-\rho_F(a_i-a_{i_0})\quad \text{(by~\eqref{estimate_1})}\\
        &\leq\rho_F(x_\ell-a_i),
		\end{array}$$ where the last estimate is valid by \cite[Theorem 6.14]{mordukhovich2023easy}. Meanwhile, one has $$\rho_F(\tilde{x}_\ell-a_i)=\rho_F(x_\ell-a_i)\; \text{ for all }\, \ell\in J\backslash L.$$ Thus, 
    	\begin{equation}\label{leq1}
        f_k^F(\tilde{x})=\max_{i\in I}\, \min_{\ell\in J}\rho_F(\tilde{x}_\ell-a_i)\leq \max_{i\in I}\, \min_{\ell\in J}\rho_F(x_\ell-a_i)=f_k^F(x).
    	\end{equation}
		Moreover, since $\tilde x_\ell\in \mathbb B_F\left[a_{i_0}, (1+\|F\|\|F^\circ\|) \rho\right]\text{ for all } \ell \in J$, one has $f_k^F(\bar x)\leq f_k^F(\tilde{x})$. Combining this with~\eqref{leq1} yields $f_k^F(\bar x)\leq f_k^F(x)$ for every $x=\left(x_1, \ldots, x_k\right) \in \mathbb{R}^{d k}$; hence $\bar x\in S^F_k$. We have thus proved that the global  optimal solution set of~\eqref{maxminn} is nonempty.

Clearly, $S_k^F=\{x\in\mathbb R^d\mid f_k^F(x)\leq\bar\alpha\}$, where $\bar\alpha$ denotes the optimal value of  problem~\eqref{max_one_center}. 
	So, by the continuity of $f_k^F(\cdot)$ we can assert that $S_k^F$ is a closed subset of $\mathbb R^d$. This completes the proof.
	\end{proof}

	
	 The next corollary is a direct consequence of Theorem~\ref{thm_global} for the case where $F$ is the closed unit ball in $\R^d$.
	\begin{corollary}\label{theo2.4}
		The global optimal solution set $S_k$ of problem~\eqref{maxmin} is nonempty and closed.
	\end{corollary}
	
A natural question arises: \textit{Are the global  optimal solution sets in Theorem~\ref{thm_global} (resp., in Corollary~\ref{theo2.4}) compact?} The affirmative answer to this question for one-center problems is given in the next theorem.

\begin{theorem}\label{one_center}
	For $k=1$, the global optimal solution set  $S_1^F$ is a nonempty compact convex subset of $\mathbb B_F\left[a_{i_0},\rho\right]$ defined in~\eqref{ball_F}, where~$\rho$ is defined by~\eqref{eqrho}.
\end{theorem}
\begin{proof} Observe that the set $S_1^F$ is nonempty by Theorem~\ref{thm_global}. Since $k=1$ by our assumption, $S_1^F$ is the optimal solution set of the optimization problem   
\begin{eqnarray}\label{max_one_center}
	\min\left\{f_1^F(x)=\max_{i\in I}\rho_F(x-a_i) \mid x\in \R^{d}\right\}.
\end{eqnarray} For  every $i\in I$, as $\rho_F(\cdot)$ is a positively homogeneous and subadditive function by~\cite[Theorem~6.14]{mordukhovich2023easy}, it is easy to show that $\rho_F(.-a_i)$ is a convex function. Hence, $$f_1^F(x)=\max\limits_{i\in I}\rho_F(x-a_i)$$ is also a convex function. This means that \eqref{max_one_center} is a convex optimization problem. Therefore, $S_1^F$ is a convex set.

Now,  fixing any $i_0 \in I$, we will show that 
\begin{equation}\label{S_F}
S_1^F\subset\mathbb B_F\left[a_{i_0},\rho\right].
\end{equation}  If $\hat x\in S_1^F$ and  $\hat x\notin \mathbb B_F\left[a_{i_0},\rho\right]$, then
\begin{equation}\label{hat_x}
	\rho_F(\hat x-a_{i_0})>\rho.
\end{equation} 
It follows from~\eqref{eqrho} and ~\eqref{hat_x} that
$$f_1^F(\hat x)=\max\limits_{i\in I}\rho_F(\hat x-a_i)>\rho=\max\limits_{i \in I}\rho_F(a_{i_0}-a_{i})=f_1^F(a_{i_0}),$$ contradicting the assumption that $\hat x$ is a global  optimal solution to  problem~\eqref{max_one_center}. We have thus proved that~\eqref{S_F} holds. Note that by Theorem~\ref{thm_global}, $S^F_k$ is closed. Combining this with~\eqref{S_F} and the compactness of $\mathbb B_F\left[a_{i_0},\rho\right]$ (see the proof of Theorem~\ref{thm_global}) yields that $S_1^F$ is a compact subset of $\mathbb R^d$.  So, we get the desired result.
\end{proof}

In the case where $k > 2$, the solution set $S_k^F$ can be either compact or noncompact. 
To justify this claim, we introduce the concept of an attraction set, which is similar to the one defined by Cuong et al.~\cite{cuong2020qualitative} for the minimum sum-of-squares clustering problem. Recently, the notion of an attraction set has been extensively used in~\cite{cuongnca,cuongoptimization} for studying the Multi-Source Weber Problem.

Consider  problem~\eqref{maxminn} and let $x=\left(x_1, \ldots, x_k\right) \in \mathbb{R}^{d k}$ be an arbitrary system of centers. We inductively construct~$k$ subsets $A_1, \ldots, A_k$ of the set of target points $A$ in the following way. Put $A_0=\emptyset$ and
$$
A_\ell=\left\{a_i \in A \backslash\left(\bigcup_{q=0}^{\ell-1} A_q\right) \mid \rho_F(x_\ell-a_i)=\min _{r\in J}\rho_F(x_r-a_i)\right\}\quad \big(\ell\in J=\{1,\ldots,k\}\big).$$ This family $\{A_1,\ldots, A_k\}$ is said to be \textit{the natural clustering} associated with $ x$.

\begin{definition}
	Let $x=\left(x_1, \ldots, x_k\right) \in \mathbb{R}^{d k}$ be a centers system. The \textit{attraction set} of a component $x_\ell$ of $x$ is the set 
	\begin{equation}\label{attraction}
		A[x_\ell]=\left\{a_i \in A \mid\rho_F(x_\ell-a_i)=\min _{r\in J}\rho_F(x_r-a_i)\right\}.
	\end{equation} If $A[x_\ell]$ is nonempty, then we say that the center $x_\ell$ is 
	\textit{attractive} with respect to the set $A$ of demand points.
\end{definition}

Clearly, one has
$A_\ell=A[ x_\ell]\setminus\left(\bigcup_{q=1}^{\ell-1}A_q\right)$ for every $\ell\in J$.

We are now prepared to show by an example that in the case $k=3$, the global optimal solution set $S^F_3$ is not necessarily compact. 

\begin{example} \label{S-noncompact} Consider problem~\eqref{maxmin} with $d=2, m=4, k=3$, and $$A=\left\{a_1=(0,0),\; a_2=(1,0),\; a_3=(0,1),\;  a_4=(1,1)\right\}.$$
Clearly, for all \(i_1, i_2 \in I\) with \(i_1 \neq i_2\), we have
\begin{equation}\label{key11}
\left\|a_{i_1} - a_{i_2}\right\|\geq 1.
\end{equation}
Suppose for a while there exists \(x=(x_1,x_2,x_3) \in \mathbb{R}^{2}\times\mathbb{R}^{2}\times \mathbb{R}^{2}\) with 
\begin{equation}\label{key12}
f_3(x) = \alpha<\frac{1}{2}, 
\end{equation} where $f_3(x)=\max\limits_{i\in I}\,\left(\min\limits_{\ell\in J}\|x_\ell-a_i\|\right)$. Since $k<m$ and $A=\bigcup\limits_{\ell=1}^k A[ x_\ell]$, there must exist $\bar\ell\in J$ such that the attraction set $A[x_{\bar\ell}]$ defined in accordance with~\eqref{attraction} contains at least two demand points. So, there are distinct indexes $\bar i_1,\bar i_2\in I$ with $a_{\bar i_1},a_{\bar i_2}\in A[x_{\bar\ell}]$. Changing the roles of $\bar i_1$ and $\bar i_2$ (if necessary), we can assume that  $\left\|x_{\bar\ell}-a_{\bar i_2}\right\|\leq\left\|x_{\bar\ell}-a_{\bar i_1}\right\|$. Thus, \begin{equation}\label{key12a}\left\|x_{\bar\ell}-a_{\bar i_2}\right\|\leq\left\|x_{\bar\ell}-a_{\bar i_1}\right\| =\min_{\ell \in J}\left\|x_\ell-a_{\bar i_1}\right\|\leq \alpha.
\end{equation}
Let \(i_1 \in I\) be given arbitrarily and let $i_2 \in I\setminus \{i_1\}$. Select an index $\ell_1\in J$ satisfying  $a_{i_1}\in A[x_{\ell_1}]$. Then, from~\eqref{key12} it follows that
\begin{equation}\label{key13}
	\left\|x_{\ell_1} - a_{i_1}\right\|=\min\limits_{\ell\in J}\|x_\ell-a_{i_1}\|\leq f(x) \leq \alpha.
\end{equation}
Combining this with \eqref{key11} and \eqref{key12} yields
\begin{equation*}
	\alpha < \frac{1}{2} < 1 - \alpha \leq \Big\|a_{i_1} - a_{i_2}\Big\| - \left\|x_{\ell_1} - a_{i_1}\right\| \leq \left\|x_{\ell_1} - a_{i_2}\right\|.
\end{equation*}
Therefore,
\begin{equation}\label{key14}
\alpha < \left\|x_{\ell_1} - a_{i_2}\right\|.
\end{equation} 

Now, setting $i_1=\bar i_1$,  $i_2=\bar i_2$, and $\ell_1=\bar\ell$, by~\eqref{key14} we get $\alpha < \left\|x_{\bar\ell} - a_{\bar i_2}\right\|$. This contradicts~\eqref{key12a}.  
Thus, we must have 
\begin{equation} \label{key15}
	f_3(x) \geq \frac{1}{2}\ \; \text{ for all }\ \;x=(x_1, x_2, x_3) \in \mathbb{R}^{2}\times\mathbb{R}^{2}\times \mathbb{R}^{2}.
\end{equation}
To show that
\begin{equation}\label{key16}
	\min \left\{f_3(x) \mid x\in \mathbb{R}^{2}\times\mathbb{R}^{2}\times \mathbb{R}^{2}\right\}  = \frac{1}{2},
\end{equation}
we choose \(\bar x_1=\left(\frac{1}{2}, 0\right)\) and \(\bar x_2=\left(\frac{1}{2}, 1\right)\).  Since
$$
	\|\bar x_1 - a_1\|= \frac{1}{2},\; \text{ } 
	\|\bar x_1 - a_2\|= \frac{1}{2},\; \text{ }	\|\bar x_2 - a_3\|= \frac{1}{2}\; \text{ and }\;
	\|\bar x_2 - a_4\|= \frac{1}{2},
$$
we have $f_3(\bar x_1, \bar x_2, x_3) \leq \frac{1}{2}$ for any $x_3\in\mathbb{R}^2$. Combining this with~\eqref{key15} gives~\eqref{key16}. Moreover, it holds that the centers system $(\bar x_1, \bar x_2, x_3)$ is a global  optimal solution to problem~\eqref{maxmin}.
Since $x_3$ can be taken arbitrarily, we conclude that
 \(S_3\) is unbounded, hence noncompact.
\end{example}

Observe that the noncompactness of \(S_3\) in Example~\ref{S-noncompact} was due to the presence of an unnecessary center, meaning that the corresponding ball did not need to cover any target points. Such an abnormal situation is also the reason for the noncompactness of \(S_k^F\).

 The next theorem provides a characterization for the compactness of the optimal solution set $S_k^F$. For the convenience of presentation, given each $k\in \mathbb{N}$,  let \(J_k= \{1, ..., k\}\). Then,
\begin{equation*}
	f^F_k(x_1, ..., x_k) = \max_{i \in I}\left( \min_{\ell \in J_k} \rho_F(x_\ell - a_i)\right)\; \ \text{for } (x_1, ..., x_k)\in \R^{dk}.
\end{equation*}
\begin{theorem}\label{compactness_condition} Consider problem \eqref{maxminn} with \(m > k \geq 2\). Then, the following properties are equivalent:
	\begin{enumerate}[label=(\roman*)]
		\item [{\rm (a)}]\(S_k^F\) is compact.
		\item [{\rm (b)}] \(\min\limits_{{(x_1, ..., x_k) \in \mathbb{R}^{dk}}}f^F_k(x_1, ..., x_k) < \min\limits_{(x_1, ..., x_{k - 1}) \in \mathbb{R}^{d(k - 1)}}f^F_{k - 1}(x_1, ..., x_{k - 1})\).
	\end{enumerate}
\end{theorem}
\begin{proof}
Observe first that $S_k^F$ is nonempty and closed by Theorem~\ref{thm_global}. Furthermore, it is obvious that
\begin{equation}\label{Vn}
f_k^F(x_1, \ldots, x_k) \leq f_{k - 1}^F(x_1, \ldots , x_{k - 1})\ \; \mbox{\rm for all }(x_1, \ldots, x_k)\in \R^{dk}.
\end{equation}
Thus, it suffices to prove that the next statements are equivalent:
	\begin{enumerate}
		\item [{\rm (a')}] \(S_k^F\) is unbounded.
		\item [{\rm (b')}] \(\min\limits_{(x_1, ..., x_k) \in \mathbb{R}^{dn}}f_k^F(x_1, ..., x_k) = \min\limits_{(x_1, ..., x_{k - 1}) \in \mathbb{R}^{d(k - 1)}}f_{k - 1}^F(x_1, ..., x_{k - 1})\).
	\end{enumerate}
(a') \(\Longrightarrow\) (b'). Put
$a = \max\{\rho_F(a_1), ..., \rho_F(a_m)\}$
and  $$ M =  \min\limits_{(x_1, ..., x_k) \in \mathbb{R}^{dk}}f^F_k(x_1, ..., x_k).$$
Then, the unboundedness \(S_k^F\) implies that there exists \(y=(y_1, ..., y_k) \in S_k^F\) such that
\begin{equation}\label{aboveinequality}
\sum_{\ell = 1}^k \|y_\ell\|^2 \geq k\|F\|^2(a + M + 1)^2.
\end{equation}
Note that $f^F_k(y_1, ..., y_k)=M$. 
By \eqref{aboveinequality}, there exists an \(\bar \ell \in J_k\) such that
\begin{equation}\label{key35}
	\|y_{\bar \ell}\| \geq \|F\|(a + M + 1).
\end{equation}
We may assume without loss of generality that \(\bar \ell = k\). Fix any \(i \in I\). By~\cite[Proposition~2.1(c)]{CW_JOGO2004} we have
$\rho_F(y_k-a_i) \geq\|F\|^{-1}\|y_k - a_i\|$ and 
$$\dfrac{\|a_i\|}{\|F\|}\leq \rho_F(a_i)\leq a.$$ Thus, using~\eqref{key35} with \(\bar \ell = k\) gives us
$$\begin{array}{lll}
	\rho_F(y_k-a_i) \geq \dfrac{\|y_k - a_i\|}{\|F\|} &\geq \dfrac{\|y_k\| - \|a_i\|}{\|F\|}& \\
	&\geq (a + M + 1) - a\\
	& > M=f_k^F(y_1, \ldots, y_k)\\ &= \max\limits_{p \in I}\left(\min\limits_{\ell \in J_k} \rho_F(y_\ell - a_p)\right)\\
	& \geq  \min\limits_{\ell \in J_k} \rho_F(y_\ell - a_i).
\end{array}$$
This implies that $\min\limits_{\ell \in J_k} \rho_F(y_\ell - a_i)= \min\limits_{\ell \in J_{n - 1}} \rho_F(y_\ell - a_i)$. Therefore,
$$\max_{i \in I} \left(\min_{\ell \in J_k} \rho_F(y_\ell - a_i)\right)= \max_{i \in I}\left(\min_{\ell \in J_{n - 1}} \rho_F(y_\ell - a_i)\right).$$ Since \(y=(y_1, ..., y_k) \in S_k^F\), we have
\begin{equation*}\begin{array}{rcl}
	\min\limits_{(x_1, ..., x_{k - 1}) \in \mathbb{R}^{d(k - 1)}}f_{k - 1}^F(x_1, ..., x_{k - 1})\leq f_{k - 1}^F(y_1, ..., y_{k - 1}) & = & \max\limits_{i \in I}\left(\min\limits_{\ell \in J_{k - 1}} \rho_F(y_\ell - a_i)\right)\\
	& = & \max\limits_{i \in I} \left(\min\limits_{\ell \in J_k} \rho_F(y_\ell - a_i)\right)\\
	& = & \min\limits_{(x_1, ..., x_k) \in \mathbb{R}^{dk}}f_k^F(x_1, ..., x_k).
\end{array}
\end{equation*}
Taking \eqref{Vn} into account, we conclude that the equality in~(b') holds.\\\\
\noindent(b') $\Longrightarrow$ (a'). Let \((y_1, ..., y_{k - 1}) \in \mathbb{R}^{d(k - 1)}\) be such that
	$$\begin{array}{ll}
		f_{k-1}^F(y_1, ..., y_{k - 1}) &= \min\limits_{(x_1, ..., x_{k - 1}) \in \mathbb{R}^{d(k- 1)}}f_{k - 1}^F(x_1, ..., x_{k - 1}) \\
		&= \min\limits_{(x_1, ..., x_k) \in \mathbb{R}^{dn}}f_k^F(x_1, ..., x_k).
	\end{array}$$
	For every \(y_k \in \mathbb{R}^d\), since $f_k^F(y_1, ..., y_k) \leq f_{k - 1}^F(y_1, ..., y_{k - 1})$, we get 
	\begin{equation*}
		f_k^F(y_1, ..., y_k) \leq \min_{(x_1, ..., x_k) \in \mathbb{R}^{dk}}f_k^F(x_1, ..., x_k),
	\end{equation*}
	and hence 
	\begin{equation*}
		f_k^F(y_1, ..., y_k) = \min_{(x_1, ..., x_k) \in \mathbb{R}^{dk}}f_k^F(x_1, ..., x_k).
	\end{equation*} 
 This means that \((y_1, \ldots , y_k) \in S_k^F\). 
	Since $y_k\in \mathbb{R}^d$ can be chosen arbitrarily, we arrive at~(a') and thus completes the proof of theorem. 
\end{proof}

Next, we study the case where $k=2$. In connection with Theorems~\ref{one_center}, \ref{compactness_condition} and Example~\ref{S-noncompact}, another question arises:  \textit{For $k=2$, is the global optimal solution set $S_2^F$ of problem~\eqref{maxminn} always compact?}

The next result answers this question affirmatively when $d= 1$.
\begin{theorem}
	Consider  problem \eqref{maxminn} with $k=2$ and $d=1$. Then,  the optimal solution set $S_2^F$ is compact.
\end{theorem}
\begin{proof} Clearly, if $m=2$, then the conclusion is obviously true. Otherwise, applying Theorem~\ref{one_center} yields that $S_1^F\neq \emptyset$. By making a translation for the whole
set $A$ (if necessary), we can assume without loss of generality that $0\in S_1^F$. Then we have $A=\{a_1,...,a_m\}\subset \mathbb B_F\left[0,r_1\right]$, where $r_1=f^F_1(0)=\max\limits_{i \in I}\rho_F(a_i)$. Set	$$
    I_{\max}=\{p\in I\mid \rho_F(a_p)=\max_{i \in I}\rho_F(a_i)\}.$$
Since the demand points are distinct by our assumption, $I_{\max}$ contains exactly two elements,
say, $I_{\max}=\{p_1,p_2\}$ for some $p_1,p_2\in I$. Assume that $a_{p_1}<0< a_{p_2}$ and $F=[a,b]$, where $a<0<b.$ Then we have
	\begin{equation*}
		\rho_F(a_{p_1})=\frac{a_{p_1}}{a}\text{ and } \rho_F(a_{p_2})=\frac{a_{p_2}}{b},
	\end{equation*} 
	which yields
	\begin{equation}\label{Imax}
\frac{a_{p_1}}{a}=\frac{a_{p_2}}{b}.
	\end{equation}
		As $m>2=k$, we can find $p_3\in I\setminus \{p_1\}$ such that
	\begin{equation}\label{xanhi}
		\rho_F(a_{p_2}-a_{p_3})=\max\limits_{i \in I\setminus\{p_1\}} \rho_F(a_{p_2}-a_i),
	\end{equation} 
	and note that
	\begin{equation}\label{a1nhoa3}
		a_{p_1}< a_{p_3}<a_{p_2}.
	\end{equation}
	Setting $\bar x_1=a_{p_1}$, $\bar x_2=\frac{ba_{p_3}-aa_{p_2}}{b-a}$ and $r_2=\frac{a_{p_2}-a_{p_3}}{b-a}$, we have
	$$\rho_F(a_{p_3}-\bar x_2)=\rho_F\left(\frac{a\left(a_{p_2}-a_{p_3}\right)}{b-a}\right)=\frac{a_{p_2}-a_{p_3}}{b-a}=\rho_F\left(\frac{b\left(a_{p_2}-a_{p_3}\right)}{b-a}\right)=\rho_F(a_{p_2}-\bar x_2).$$
	This together with~\eqref{xanhi} implies that
	$A\setminus \{a_{p_1}\}\subset \mathbb B_F[\bar x_2,r_2]$, and hence $$A\subset \mathbb B_F[\bar x_1,r_2]\cup\mathbb B_F[\bar x_2,r_2].$$
	Moreover, it follows from~\eqref{a1nhoa3} and~\eqref{Imax} that
	$$\frac{a_{p_2}-a_{p_3}}{b-a}< \frac{a_{p_2}-a_{p_1}}{b-a} = \frac{a_{p_2}}{b}.$$ 
	Thus, we have
	$$\min_{(x_1,x_2)\in \R\times \R}f_2^F(x_1,x_2)\leq \frac{a_{p_2}-a_{p_3}}{b-a}<\frac{a_{p_2}}{b}=\min_{x_1\in\mathbb{R}}f_1^F(x_1).$$
	Using Theorem~\ref{compactness_condition}, we see that $S_2^F$ is compact. The proof is now complete. 
\end{proof}	

The following example gives a negative answer to the question above in the case where  $\rho_F$ is the maximum norm.
\begin{example} \label{vdS2noncompact}
Let $d=2$, $A=\left\{a_1=(0,0), a_2=(1,0), a_3=(0,1), a_4=(1,1)\right\}$, and let \( F \) be the square with vertices $(-1,-1),  (-1,1), (1,-1)$, and $ (1,1)$. Consider the problem
$$
	\text{min}\left\{f_1^F(x_1)=\max_{i\in I}\rho_F(x_1-a_i)\mid x_1=(x_1^1,x_1^2)\in \R^{2}\right\},$$
which reads as
\begin{eqnarray}\label{maxmin1}
	\text{min}\left\{f_1^F(x_1)=\max_{i\in I}\|x_1-a_i\|_\infty\mid x_1=(x_1^1,x_1^2)\in \R^{2}\right\},
\end{eqnarray}
where $\|\cdot\|_\infty$ is the maximum norm. The optimal value of problem~\eqref{maxmin1} is $\frac{1}{2}$. Indeed, suppose that 
$$\min_{x_1\in \mathbb{R}^2}f_1^F(x_1)=f_1^F(\bar x_1)=r_1\text{  for some }\bar x_1\in \mathbb{R}^2.$$
Since $a_1,a_4\in \mathbb B_F\left[\bar x_1, r_1\right]$, we have
$$\|a_1-a_4\|_\infty\leq \|a_1-\bar x_1\|_\infty+ \|\bar x_1-a_4\|_\infty\leq r_1+r_1=2r_1.$$ Combining this with the fact that $\|a_1-a_4\|_\infty=1$ yields $r_1 \geq \frac{1}{2}$. By choosing $\bar x_1=\left(\frac{1}{2},\frac{1}{2}\right)$, we get $A \subset \mathbb  B_F\left[\bar x_1, \frac{1}{2}\right]$. Thus, the global optimal value of~\eqref{maxmin1} is~$\frac{1}{2}$ and $\left(\frac{1}{2},\frac{1}{2}\right)$ is a global optimal solution. Next, we consider the problem
\begin{eqnarray}\label{maxmin2}
	\text{min}\left\{f_2^F(x_1,x_2)=\max_{i\in I}\left(\min_{\ell\in J}\|x_\ell-a_i\|_\infty\right)\mid (x_1,x_2)\in \left(\R^{2}\right)^2\right\}
\end{eqnarray} 
and let $(\bar x_1,\bar x_2)$ be a global  optimal solution of it. 
Suppose for a while that the optimal value of~\eqref{maxmin2}, which is denoted by~$r_2$, is less than $\frac{1}{2}$. Then we can assume without loss of generality that $\mathbb B_F\left[\bar x_1,r_2\right]$ contains at least two distinct demand points, which are denoted by $a_{i_1}$ and $a_{i_2}$. Then we have
\begin{equation*}\label{greater}
	2r_2<1=\left\|a_{i_1} - a_{i_2}\right\|_\infty\leq 2r_2,
\end{equation*}
a contradiction. Thus, 
$$\min_{(x_1,x_2)\in \left(\mathbb{R}^2\right)^2}f^F_2(x_1,x_2)=\frac{1}{2}=\min_{x_1\in \mathbb{R}^2}f^F_1(x_1).$$
Applying Theorem~\ref{compactness_condition}, we can assert that $S_2^F$ is noncompact.
\end{example}
Next, we will show that when $k=2$, the global optimal solution set $S_2$ of problem \eqref{maxmin} is always compact. To proceed, we need the lemma below.

\begin{lemma}\label{existence_subspace_V} Given any finite subset $A=\{a_1,...,a_m\}$ of $\mathbb{R}^d$ with $d\geq 1$, one can find a \((d - 1)\)-dimensional linear subspace \(V\) of \(\mathbb{R}^d\) such that 
\begin{equation}\label{intersection-empty}
   \left (V \setminus \{0\}\right) \cap \{a_1, ..., a_m\}= \emptyset.
\end{equation}
\end{lemma}
\begin{proof}  If $A$ consists just of the vector 0, then any \((d - 1)\)-dimensional linear subspace \(V\) of \(\mathbb{R}^d\) would satisfy the condition~\eqref{intersection-empty}. Now, suppose that $A$ contains at least one nonzero vector. For every $a_i\in A\setminus\left\{0\right\}$, define
$B_i=\{v\in \R^n\mid \la v, a_i\ra=0\}$
and note that $B_i$ is a closed subset of set $\mathbb{R}^d$ with empty interior. Thus, the latter set is of the first category in $\mathbb{R}^d$ (see, e.g.,~\cite[p.~42]{Rudin}). Hence the set 
$B=\bigcup_{i\in I,\, a_i\neq 0} B_i$,
where $I=\{1,\ldots,m\}$, is also of the first category in $\mathbb{R}^d$. As the complete metric space $\mathbb{R}^d$ is of the second category in itself (see~\cite[p.~42]{Rudin}), there exists $\bar v\in \mathbb{R}^d\setminus B$. Then we have that $\bar v$ is nonzero and $\la \bar v, a_i\ra\neq 0$ for every $i\in I$ with $a_i\neq 0$. Therefore, the \((d - 1)\)-dimensional linear subspace $V=\{x\in \R^n\mid  \la \bar v, x\ra=0\}$ of \(\mathbb{R}^d\) satisfies~\eqref{intersection-empty}.
\end{proof}

\begin{remark}{\rm In the proof of Lemma~\ref{existence_subspace_V}, since each set $B_i$ is a hyperplane, it has Lebesgue measure zero. Consequently, the set $B$ also has Lebesgue measure zero. This provides an alternative way to the existence of a vector $\bar{v} \in \R^d \setminus B$.}
   
\end{remark}

We are now ready to prove that $S_2$ is compact.

\begin{theorem}\label{compactwhen_k=2}
Consider  problem~\eqref{maxmin} with $k=2$. Then $S_2$ is compact.
\end{theorem}
\begin{proof}
Theorem~\ref{one_center}, we have $S_1\neq \emptyset$. Without loss of generality, assume $0\in S_1$. Let \(r= f_1(0)\). According to Lemma~\ref{existence_subspace_V}, there is a \((d - 1)\)-dimensional linear subspace \(V\) of \(\mathbb{R}^d\) such that~\eqref{intersection-empty} holds. Let \(w \in V^\perp\),  where
$V^\perp$ is the orthogonal complement of $V$, be such that \(\|w\| = 1\). Fix any \(a \in \{a_1, ..., a_m\}\). 

If $a=0$, then for any $\varepsilon \in \left( 0,\frac{r}{\sqrt{2}}\right]$ we have
\begin{equation}\label{choosee1} a\in \mathbb B\left[\epsilon w, \sqrt{r^2 - \epsilon^2}\right]. 
 \end{equation}
If $a \neq 0$, then by ~\eqref{intersection-empty} we see that \(a \notin V\). Note that $V^\perp$ is a 1-dimensional subspace, and so we get $\langle a, w\rangle \neq 0$. Thus, for any  \(\epsilon\in (0,|\langle a, w\rangle|)\) we have
$$	\begin{array}{ll}
		\|a - \operatorname{sgn}(\langle a, w\rangle)\epsilon w\|^2 &= \|a\|^2  - 2|\langle a, w\rangle|\epsilon + \epsilon^2 \\
		&< \|a\|^2 - 2\epsilon^2 + \epsilon^2 \\
		&\leq r^2 - \epsilon^2.
	\end{array}$$
This implies that
\begin{equation}\label{keycuoi}
    a\in \Theta=\left\{\begin{array}{ll}
      \mathbb B\left[\epsilon w, \sqrt{r^2 - \epsilon^2}\right] &  \text{\rm if } \langle a, w \rangle > 0,\\[0.1in]
    \mathbb B\left[-\epsilon w, \sqrt{r^2 - \epsilon^2}\right]  & \text{\rm otherwise.}
 \end{array}\right.
\end{equation}
 Now, for every \(i = 1, ..., m\), set
$$\epsilon_i=\begin{cases}
     \frac{r}{\sqrt{2}}&\textrm { if } a_i = 0,\\[0.1in]
     \frac{|\langle a_i, w\rangle|}{2} &\textrm { if } a_i \neq 0.
\end{cases}$$
 Define \(\bar\epsilon > 0\) by  $\bar\epsilon= \min\{\epsilon_1, ..., \epsilon_m\}.$ Then, due to~\eqref{choosee1} and~\eqref{keycuoi}, for this $\epsilon$ we have $$\{a_1, ..., a_m\} \subset  \mathbb B\left[\bar\epsilon w, \sqrt{r^2 - \bar\epsilon^2}\right] \cup \mathbb B\left[-\bar\epsilon w, \sqrt{r^2 - {\bar\epsilon}^2}\right].$$ This implies that
$$\min_{(x_1, x_2) \in \mathbb{R}^{d}\times \mathbb{R}^{d}}f_2(x_1, x_2) \leq \sqrt{r^2 - {\bar\epsilon}^2} < r = \min_{x_1 \in \mathbb{R}^d}f_1(x_1).$$
 So, according to Theorem~\ref{compactness_condition}, the set $S_2$ is compact. 
\end{proof}

The next example shows that $S_2$ can be disconnected, and hence nonconvex.
\begin{example} \label{S-compact}
	Consider problem~\eqref{maxmin} with $d=1, m=3, k=2$, and $$A=\left\{a_1=0,\; a_2=1,\; a_3=10\right\}.$$
	Since $|a_{i_1}-a_{i_2}|\geq 1$ for all $i_1,i_2\in I$ satisfying $i_1\neq i_2$, by using the same argument as in Example~\ref{S-noncompact}, we can show that
	\begin{equation}\label{key34}
		\min_{(x_1, x_2) \in \mathbb{R}^2} f_2(x_1, x_2) = \frac{1}{2}.
	\end{equation}
	Take any $(\bar x_1,\bar x_2)\in S_2$. Since $k<m$, by \eqref{key34}, we may assume without loss of generality that
	$\mathbb B\left[\bar x_1,\frac{1}{2}\right]$ contains at least one demand points. Then, a direct verification shows that $\bar x_1=\frac{1}{2}$ and \(\bar x_2 \in \left[a_3 - \frac{1}{2}, a_3 + \frac{1}{2}\right]\). It follows that
	\begin{equation*}
		S_2 = \left(\left\{\frac{1}{2}\right\} \times \left[a_3 - \frac{1}{2}, a_3 + \frac{1}{2}\right]\right) \bigcup \left(\left[a_3 - \frac{1}{2}, a_3 + \frac{1}{2}\right] \times \left\{\frac{1}{2}\right\}\right),
	\end{equation*}
	and hence $S_2$ is  compact and disconnected.
\end{example}



		\section{Local Optimal Solutions}
		
For $k\geq 2$, the $k$-center problem~\eqref{maxminn} is a nonconvex optimization problem in general. So, it may have local optimal solutions, which are not global optimal ones. In this section, we will establish a sufficient condition for the existence of such local optimal solutions. We first recall the  notion of {\em local optimal solution} to  problem~(\ref{maxminn}).

	\begin{definition}
		We say that $\bar x = (\bar {x}_1, \ldots, \bar {x}_k) \in  \mathbb{R}^{dk}$ is a \textit{local optimal solution} to (\ref{maxminn})
		if there is $\epsilon >0$ such that $f_k^F(\bar x) \leq  f_k^F(x)$ for every $x = (x_1, \ldots, x_k) \in  \mathbb{R}^{dk}$ satisfying
		$\|x_\ell - \bar x_\ell\| < \epsilon$ for all $\ell\in J$. The set of the local optimal solutions to~(\ref{maxminn}) (resp., to~(\ref{maxmin})) is denoted by $S^{F,\, {\rm loc}}_k$ (resp., $S^{\rm loc}_k$).
	\end{definition}

The next example shows that the inclusion $S_k \subset S^{\rm loc}_k$ can be strict.
	\begin{example}
		Consider problem~\eqref{maxmin} with $d=1, m=2, k=2$,  $A=\left\{a_1=0,\; a_2=1\right\}$. We have $S_2=\left\{\left(a_1, a_2\right),\left(a_2, a_1\right)\right\}$ and the global optimal value of problem is $0$. The vector $\bar{x}=\left(\bar{x}_1, \bar{x}_2\right)= \left(\frac{1}{2},3\right)$  is a local optimal solution, while it is not a global  optimal solution. Indeed, for $\varepsilon=\frac{1}{2}$, a direct verification shows that $f_2(\bar{x})\leq f_2(x)$ for every $x=\left(x_1, x_2\right) \in \mathbb{R}^2$ satisfying $|x_1-\bar{x}_1|<\varepsilon$ and $|x_2-\bar{x}_2|<\varepsilon$. Since $f_2(\bar{x})=\frac{1}{2}$, we have  $\bar{x} \in S^{{\rm loc}}_2 \setminus S_2$.
	\end{example}

	For $x=(x_1,\ldots,x_k)\in \mathbb{R}^{dk}$, we have
	$$f_k^F(x)=\max_{i\in I}\left(\min_{\ell\in J}\rho_F(x_\ell-a_i)\right)=\max_{i\in I}\left(\sum_{r\in J}\rho_F(x_r-a_i)- \max_{\ell\in J} \sum_{r\in J\setminus \{\ell\}}\rho_F(x_r-a_i)\right ).
	$$
	Define the functions
	\begin{equation}\label{key1}
		g_i(x)=\sum_{r\in J}\rho_F(x_r-a_i),    
	\end{equation}
	
	\begin{equation}\label{key2}
		h_i(x)=\max_{\ell\in J} \sum_{r\in J\setminus\{\ell\}}\rho_F(x_r-a_i),    
	\end{equation}
	and
		\begin{equation}\label{key2a} h_{i,\ell}(x)=\sum_{r\in J\setminus \{\ell \}} \rho_F(x_r-a_i)\; \text{ for every } (i,\ell)\in I\times J.
	\end{equation}
	Then, we obtain the representation
	\begin{equation}\label{key3}
		f_k^F(x)=\max_{i\in I}\,(g_i(x)-h_i(x)),   
	\end{equation}
	where
	\begin{equation}\label{key7}
		h_i(x)=\max_{\ell\in J}h_{i,\ell}(x).
	\end{equation}
	Set
		\begin{equation}\label{key5}
		J_i(x)=\{\ell\in J\mid h_i(x)=\max_{\ell\in J}h_{i,\ell}(x)\}.
		\end{equation}
		The following proposition shows that for any $(i,\ell)\in I\times J$, the inclusion $\ell\in J_i(x)$ means that the center $x_\ell$ is attractive to the demand point $a_i$.
	
	\begin{proposition} \label{prop34}
		For every $i\in I$, where $J_i(x)$ defined by~\eqref{key5}, one has
			\begin{equation}\label{J(i)}
				J_i(x)=\{\ell\in J\mid a_i\in A[x_\ell]\},
			\end{equation}
			where $A[x_\ell]$ is the attraction set defined by (\ref{attraction}).
			\end{proposition}
	\begin{proof}
	    It follows from~\eqref{key2a} that
		$$h_{i,\ell}(x)= \sum_{r\in J}	\rho_F(x_r-a_i)-\rho_F(x_\ell-a_i)\; \text{ for every } (i,\ell)\in I\times J .$$
		So, by~(\ref{key2}) one has $$\begin{array}{ll}
			h_i(x)&=\max\limits_{\ell\in J} \left(\sum\limits_{r\in J}\rho_F(x_r-a_i)-\rho_F(x_\ell-a_i)\right)\\
			&=\sum\limits_{r\in J}\rho_F(x_r-a_i)+ \max\limits_{\ell\in J}\left (-\rho_F(x_\ell-a_i)\right)\\
			&=\sum\limits_{r\in J}\rho_F(x_r-a_i)- \min\limits_{\ell\in J}\rho_F(x_\ell-a_i).
		\end{array}$$
		Thus, the maximum in~(\ref{key7}) is achieved if and only if the minimum $\min\limits_{\ell\in J}\rho_F(x_\ell-a_i)$ is
		attained. Together with~(\ref{attraction}) and~(\ref{J(i)}),  this fact implies~(\ref{J(i)}).
	\end{proof}

	Now we are already able to establish a sufficient condition for the existence of local optimal solutions to problem \eqref{maxminn}.
	
	\begin{theorem}\label{theo_sufficient} 	For $\bar x=(\bar x_1,\ldots, \bar x_k)\in \mathbb{R}^{dk}$, assume that
	 $J_i(\bar x)$ is a singleton for every $i\in I$. Assume further that for every $ \ell\in J$, if  $A[\bar{x}_\ell]$ is nonempty, then $\bar{x}_\ell$ is a global optimal solution of the 1-center problem defined by the data set $A[\bar{x}_\ell]$. Then the vector $\bar x=(\bar x_1,\ldots, \bar x_k)$ is a local optimal solution to problem $(\ref{maxminn})$.
	\end{theorem}
	
	\begin{proof}
	By the hypothesis, for every $i\in I$, we denote the unique element of $J_i(\bar{x})$ by $\ell(i)$. Then one has
	\begin{equation}\label{keycc}
\rho_F(\bar{x}_{\ell(i)}-a_i)<\rho_F(\bar{x}_{\ell}-a_i)\text{ for each }\ell \in J \setminus\{\ell(i)\}.
	\end{equation}
	 Thus, there exists $\varepsilon>0$ such that
		\begin{equation}\label{key25}
			\rho_F(\widetilde{x}_{\ell(i)}-a_i)<\rho_F(\widetilde{x}_{\ell}-a_i) \quad \text{for all }i \in I \text{ and all } \ell \in J \setminus\{\ell(i)\},    
		\end{equation}
		whenever the vector $\widetilde{x}=\left(\widetilde{x}_1, \ldots, \widetilde{x}_k\right) \in \mathbb{R}^{d k}$ satisfies the condition $\|\widetilde{x}_\ell-\bar{x}_\ell\|<\varepsilon$ for all $\ell \in J$. Then it follows from~(\ref{key25}) that $J_i(\widetilde{x})=\{\ell(i)\}$ for every $i\in I$, which yields
		\begin{equation}\label{key37}
			\rho_F(\widetilde{x}_{\ell(i)}-a_i)=\min _{\ell\in J }\rho_F(\widetilde{x}_\ell-a_i).
		\end{equation}
			Since $J_i(\bar{x})=\{\ell(i)\}$ and $J_i(\widetilde{x})=\{\ell(i)\}$, we obtain from~\eqref{J(i)} that
				\begin{equation}\label{key51}
				\ell(i)=\ell\text{ for every }i\in I \text{ such that }\,a_i\in A[\bar{x}_\ell].
			\end{equation}
			and
			\begin{equation}\label{key52}
\ell(i)=\ell\text{ for every }i\in I\text{ such that }\,a_i\in A[\widetilde{x}_\ell].
			\end{equation}
	Moreover, for every $\ell \in J$ satisfying $A\left[\bar{x}_\ell\right] \neq \emptyset$, we get from the hypothesis that
		\begin{equation}\label{key53}
		\max_{i \in I,\; a_i\in A[\bar x_\ell]} \rho_F(\widetilde{x}_\ell-a_i)\geq \max_{i \in I,\; a_i\in A[\bar x_\ell]} \rho_F(\bar{x}_\ell-a_i).
		\end{equation}
	Thus, we have
	$$
		\begin{aligned}
			f_k^F(\widetilde{x})=\max\limits_{i \in I}\left(\min _{\ell \in J}\rho_F(\widetilde{x}_\ell-a_i)\right)
			& =\max_{i \in I}\rho_F(\widetilde{x}_{\ell(i)}-a_i) \quad (\text{\rm by }~\eqref{key37})\\
			& =\max_{\ell \in J,\,A[\widetilde x_\ell]\neq \emptyset}\left(\max_{i \in I,\; a_i\in A[\widetilde x_\ell]} \rho_F(\widetilde{x}_{\ell(i)}-a_i)\right) \\
			& =\max_{\ell \in J,\,A[\widetilde x_\ell]\neq \emptyset}\left(\max_{i \in I,\; a_i\in A[\widetilde x_\ell]} \rho_F(\widetilde{x}_{\ell}-a_i)\right) \quad (\text{\rm by }~\eqref{key52})\\
			& = \max_{\ell \in J,\,A[\bar x_\ell]\neq \emptyset}\left(\max_{i \in I,\; a_i\in A[\bar x_\ell]}\rho_F(\widetilde{x}_\ell-a_i)\right) \\
			& \geq \max_{\ell \in J,\,A[\bar x_\ell]\neq \emptyset}\left(\max_{i \in I,\; a_i\in A[\bar x_\ell]}\rho_F(\bar{x}_\ell-a_i)\right) \quad (\text{\rm by }~\eqref{key53})\\
			& = \max_{\ell \in J,\,A[\bar x_\ell]\neq \emptyset}\left(\max_{i \in I,\; a_i\in A[\bar x_\ell]} \rho_F(\bar{x}_{\ell(i)}-a_i)\right) \quad (\text{\rm by }~\eqref{key53})\\
			&=\max_{i \in I}\rho_F(\bar{x}_{\ell(i)}-a_i)\\
			&=\max\limits_{i \in I}\left(\min _{\ell \in J}\rho_F(\bar{x}_\ell-a_i)\right) \quad (\text{\rm by }~\eqref{keycc})\\
			&=f_k^F(\bar{x}),
		\end{aligned}
		$$
		which implies that $\bar{x}=\left(\bar{x}_1, \ldots, \bar{x}_k\right)\in S_k^{F, \, {\rm loc}}$. Note that in the proof above we use the obvious fact that $A=\cup_{\ell\in J}\cup_{a_i\in A[x_\ell]}$ for every $x=(x_1,..,x_k)\in \mathbb R^{dk}$. The proof is now complete.
	\end{proof}
Next, we use Example~\ref{S-compact} to illustrate Theorem~\ref{theo_sufficient}.
 	\begin{example}
		Consider Example~\ref{S-compact}.
	For $\bar x=(\bar x_1,\bar x_2)=(5,30),$ it can be verified that $J_1(\bar x)=\{1\}$, $J_2(\bar x)=\{1\}$, $J_3(\bar x)=\{1\}$, $A[\bar x_1]=\left\{a_1,a_2,a_3\right\}$ and $A[\bar x_2]=\emptyset.$ Moreover, $\bar x_1$ is a global optimal solution to the $1$-center problem defined by the data set  $A[\bar x_1]=\left\{a_1,a_2,a_3\right\}$. Thus, all assumptions of  Theorem~\ref{theo_sufficient} are satisfied. According to this theorem,
 $(\bar x_1,\bar x_2)=(5,30)$ is a local optimal solution to 
problem~\eqref{maxmin} and the local optimal value is~$5$.
 Note that $(\bar x_1,\bar x_2)=(5,30)$ is not a global optimal solution to 
problem~\eqref{maxmin} because the global optimal value is $\frac{1}{2}$, as shown in Example~\ref{S-compact}.
\end{example}

	\begin{remark}
			It should be emphasized that Theorem~\ref{theo_sufficient} may not be necessary, in general,
		for the existence of a local solution to  problem~\eqref{maxminn}, as  demonstrated
		by the following example.
	\end{remark}
	\begin{example}
		Consider problem~\eqref{maxmin} with $d=2, m=3, k=2$, and $$A=\left\{a_1=(0,0),\; a_2=(1,0),\; a_3=(0,1)\right\}.$$
		It can be verified that $\bar x=(\bar x_1,\bar x_2)\in \mathbb{R}^2\times\mathbb{R}^2$, where $\bar x_1=\left(\frac{1}{2},\frac{1}{2}\right)$ and $\bar x_2=\left(-\frac{1}{100},-\frac{1}{100}\right)$, is a
		local optimal solution to problem~\eqref{maxmin}. However, $\bar{x}_2$ is not a global optimal solution to the $1$-center problem defined by the data set   $A[\bar{x}_2]=\{ a_1\}$. Note that $(\bar x_1,\bar x_2)$ is not a local optimal solution to  the multi-source Weber problem studied in \cite{{cuongoptimization}}.
		\end{example}
\section{Conclusions}
	After introducing the $k$-center problem, we established a solution existence theorem and provided a comprehensive study of the fundamental properties of global solutions, supplemented by several illustrative examples. We then explored the existence of local solutions. The results obtained can be used
to analyze existing solution methods for the $k$-center problem.		
	

\begin{thebibliography}{10}

	\bibitem{ANQ} An, N. T., Nam, N. M., and Qin, X.: Solving $k-$center problems involving sets based on optimization techniques. J. Global Optim. 76, 189–209 (2020).
		
		\bibitem{Beck}
		Beck, A., and Sabach, S.: Weiszfeld's method: old and new results. J. Optim. Theory Appl. 164 (2015), 1--40.
		
		
		\bibitem{Brimberg}
		Brimberg, J., Mladenovi\'{c}, N.: Degeneracy in the multi-source Weber problem. Math. Program. 85 (1999), Ser. A, 213--220 
		
		
		\bibitem{CW_JOGO2004} Colombo, G., and Wolenski, P. R.: The subgradient formula for the minimal time function in the case of constant dynamics in Hilbert space. J. Glob. Optim. 28 (2004), 269--282.
		
		
		\bibitem{cuongnca}
		Cuong, T.~H., Thien, N.~V., Yao, J.-C., and Yen, N.~D.:
		Global solutions of the multi-source Weber problem.
		J. Nonlinear Convex Anal. 24 (2023), 669--680.
		
		\bibitem{cuongoptimization}
		Cuong, T.~H., Wen, C.-F., Yao, J.-C., and Yen, N.~D.:		
		Local solutions of the multi-source Weber problem, Optimization, https://doi.org/10.1080/02331934.2024.2331798, Published online: 21 March 2024.
		
		\bibitem{cuong_jnca}
		Cuong, T.~H., Wen, C.-F., Yao, J.-C., and Yen, N.~D.: Stability analysis of the multi-source Weber problem, Journal of Nonlinear and Convex Analysis (accepted for publication). 
		
		
		\bibitem{cuong2020qualitative} Cuong, T.~H., Yao, J.-C., and Yen, N.~D.:
		Qualitative properties of the minimum sum-of-squares clustering
		problem. Optimization 69 (2020), 2131--2154.
		
				
		\bibitem{Gunther1}		
		Gunther, C., and Tammer, C.: Relationship between constrained and unconstrained multi-objective optimization and application in location theory. Math. Methods Oper. Res. 84 (2016), 359--387.
		
		\bibitem{Gunther2}
		Gunther, C., Tammer, C., Hillmann, M.,   and Winkler, B.: Project ``Facility location optimizer'', The Martin Luther
		University of Halle-Wittenberg, 2018.
		
		
		
		\bibitem{longoptimletter}
		Long, V.~S.~T.:
		A new notion of error bounds: necessary and sufficient conditions.
		Optim. Lett. 15 (2021), 171--188.
		
		
		
		\bibitem{mordukhovich2011applications}
		Mordukhovich, B. S., and Nam, N.~M.:
		Applications of variational analysis to a generalized
		Fermat-Torricelli problem.	J. Optim. Theory Appl. 148 (2011),
		431--454.
		
		\bibitem{mordukhovich2019fermat}
		Mordukhovich, B., and Nam, N.~M.:
		The Fermat-Torricelli problem and Weiszfeld's algorithm in the	light of convex analysis.
		J. Appl. Numer. Optim. 1 (2019), 205--215.
		
		\bibitem{mordukhovich2023easy}
		Mordukhovich, B., and Nam, N.~M.:
		An easy path to convex analysis and applications.
		Second edition, Springer Nature, 2023.
		
		
		\bibitem{nam2014constructions}
		Nam, N.~M., Hoang, N., and An, N.~T.:
		Constructions of solutions to generalized Sylvester and
		Fermat--Torricelli problems for euclidean balls.
		J. Optim. Theory Appl. 160 (2014),
		483--509.
		
		
		\bibitem{Rudin}	Rudin, W.,: Functional Analysis,
		Second edition, McGraw-Hill, Inc., New York, 1991.
  
		\bibitem{Sabach}
		Sabach, S., Teboulle, M., and Voldman, S.: A smoothing alternating minimization-based algorithm for clustering with sum-min of Euclidean norms. Pure Appl. Funct. Anal. 3 (2018),
		 653–-679. 
		
		
		
		

		\bibitem{syl}
		Sylvester, J.~J.:
		A question in the geometry of situation.
		Quarterly J. Pure Appl. Math.~1 (1857),
		79--80.

		\bibitem{tuy2016convex}
		Tuy, H.
	 Convex analysis and global optimization.
	 Second edition, Springer, 2016.
		
		
		
		
		
		
	\end{thebibliography}
	\end{document}